\documentclass{article}

\usepackage{arxiv}

\usepackage{amsfonts, amsthm}

\usepackage[utf8]{inputenc} 
\usepackage[T1]{fontenc}    
\usepackage{bookmark}       
\usepackage{url} 
\usepackage{graphicx}

\newtheorem{proposition}{Proposition}
\newtheorem{remark}{Remark}
\newtheorem{theorem}{Theorem}

\title{Means in money exchange operations }

\author{ 
	\href{https://orcid.org/0000-0003-0480-1852}{\includegraphics[scale=0.06]{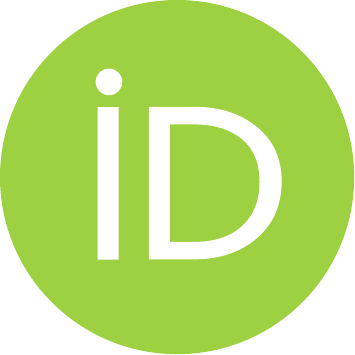}\hspace{1mm}Jacek Bojarski} \\
	Institute of Mathematics \\
	University of Zielona G\'{o}ra \\
	Szafrana 4A, PL 65-516 Zielona G\'{o}ra, Poland \\
	\texttt{j.bojarski@wmie.uz.zgora.pl} \\
	Cinkciarz.pl \\
	Sienkiewicza 9, PL 65-001 Zielona Góra, Poland \\
	\texttt{j.bojarski@cinkciarz.pl}
	\And
	\href{https://orcid.org/0000-0000-0000-0000}{\includegraphics[scale=0.06]{orcid.pdf}\hspace{1mm}Janusz Matkowski} \\
	Institute of Mathematics \\
	University of Zielona G\'{o}ra \\
	Szafrana 4A, PL 65-516 Zielona G\'{o}ra, Poland \\
	\texttt{j.matkowski@wmie.uz.zgora.pl}
}

\begin{document}
	
\maketitle	

\begin{abstract}
It is observed that in some money exchange operations, the applied $n$%
-variable mean $M$ should be self reciprocally-conjugate, i.e. it should
satisfy the equality 
\[
M\left( x_{1},\ldots,x_{n}\right) M\left( \frac{1}{x_{1}},\ldots,\frac{1}{x_{n}}
\right) =1,\quad x_{1},\ldots,x_{n}>0. 
\]
The main result says that the only weighted quasiarithmetic mean satisfying
this condition is the weighet geometric mean.
\end{abstract}



\section{Introduction}\label{sec:1}


Motivated by some money exchange operations, we consider the $n$-variable
means $M$ acting in the interval $\left( 0,\infty \right) $ and satisfying
the condition 
\begin{equation}\label{eq:1}
M\left( x_{1},\ldots,x_{n}\right) M\left( \frac{1}{x_{1}},\ldots,\frac{1}{x_{n}}%
\right) =1,\quad x_{1},\ldots,x_{n}>0. 
\end{equation}%
The role of this condition, on the model of work of two currency market
analysts, acting in two different countries, is explained in Section \ref{sec:2}. In
Section \ref{sec:3} we recall the basic notions concerning means (\cite{Bullen}). For
a homeomorphic mapping $\varphi $ and a mean $M$ on $(0,\infty)$, we define 
$M^{\left[ \varphi \right]}$, a $\varphi$-conjugate mean, and we
remark that a mean $M$ satisfies condition (1), iff $M$ is self-reciprocally
conjugate, which holds true, if and only if $M^{\left[ \exp \right]}$ is an
odd mean on $\mathbb{R}$ (Remark \ref{rem:3}). In Section \ref{sec:4} we recall some basic facts
on weighted quasiarithmetic means. In Section \ref{sec:5} we determine the form of all
odd weighted quasiarithmetic means in $\mathbb{R}$ (Proposition \ref{prop:1}), and
then, making use of Remark \ref{rem:3}, we establish the form of all weighted
quasiarithmetic means in $\left( 0,\infty \right) $ satisfying condition (\ref{eq:1})
(Theorem \ref{the:1}). Assuming homogeneity of the mean $M$, a "good" property of
a mean, this result leads to Theorem \ref{the:2}, which says that the geometric
weighted mean $M=$ $\mathcal{G}_{p_{1},\ldots,p_{n}},$%
\[
\mathcal{G}_{p_{1},\ldots,p_{n}}\left( x_{1},\ldots,x_{n}\right)
:=x_{1}^{p_{1}}\cdot \ldots\cdot x_{n}^{p_{n}}, \quad x_{1},\ldots,x_{n}>0 \nonumber
\]
is the only homogeneous $n$-variable weighted quasiarithmetic mean with
weights $p_{1},\ldots,p_{n}>0$, $p_{1}+\ldots+p_{n}=1$, satisfying condition $(\ref{eq:1})$.

Our considerations show that the geometric mean is the most proper tool in
the money exchange operations.

\section{Motivation} \label{sec:2}

To explain the problem, consider, for example two currency market analysts
in Great Britain and the United States, dealing with GBP and USD. For the
sake of clarity, let us assume that the bid and ask rates are the same and
no rounding is applied.

The UK analyst analyses the USD/GBP rate (denoted by $x$), where USD is the
base currency and GPB is the quote currency. The analyst in the United
States analyses the rate of the same pair, but in a different ordering
GBP/USD (denote it by $y$), where GBP is the base currency and USD is the
quote currency. Then, to avoid arbitration, for any given time $t$ the
condition $x_{t}\cdot y_{t}=1$ must be met. Suppose further that both
analysts record the rates at the same times $t_{1}$, $t_{2}$, $\ldots $, $%
t_{n}$. Then, in Great Britain, the analyst will observe USD/GBP exchange
rates
\[
\left( x_{t_{1}},x_{t_{2}},\ldots ,x_{t_{n}}\right) ,
\]
and in the United States, GBP/USD exchange rates will be: 
\[
\left( y_{t_{1}},y_{t_{2}},\ldots ,y_{t_{n}}\right) =\left( \frac{1}{%
x_{t_{1}}},\frac{1}{x_{t_{2}}},\ldots ,\frac{1}{x_{t_{n}}}\right).
\]

Note that based on their observations, analysts know what the other analyst
is observing. Analysts use many parameters/indicators to facilitate
inference. The basic one, frequently used, is the arithmetic mean of the
course values (hourly, daily, weekly, etc.). But it is easy to see that the
equality 
\[
\frac{1}{n}\sum_{j=1}^{n}x_{t_{j}}=\frac{1}{\frac{1}{n}\sum_{j=1}^{n}\frac{1}{x_{t_{j}}}}
\]
holds if and only if $x_{t_{1}}=\ldots=x_{t_{n}},$ which means that analysts
knowing their own arithmetic mean do not know what the mean of the other
analyst is.

In this situation appears a natural problem of characterization of means $M$
satisfying the relationship 
\[
M\left( x_{1},x_{2},\ldots ,x_{n}\right) =\frac{1}{M\left( \frac{1}{x_{1}},%
\frac{1}{x_{2}},\ldots ,\frac{1}{x_{n}}\right) },\quad \;n\geq 2,\;x_{i}>0,\;
i=1,2,\ldots ,n.
\]

\section{Some preliminaries and remarks on condition~(\ref{eq:1})} \label{sec:3}

Let $I\subset \mathbb{R}$ be an interval and let $n\in \mathbb{N}$, $n\geq 2$
be fixed. \\
Recall that a function $M:I^{n}\rightarrow I$ is called an $n$-\textit{%
variable mean} in $I$, if 
\[
\min \left( x_{1},\dots,x_{n}\right) \leq M\left( x_{1},\ldots,x_{n}\right) \leq
\max \left( x_{1},\ldots,x_{n}\right), \quad \left(
x_{1},\ldots,x_{n}\right) \in I^{n},
\]
and the mean is called \textit{strict,} if these inequalities are sharp for
all nonconstant sequences $\left( x_{1},\ldots,x_{n}\right) \in I^{n}$.

A mean $M:I^{n}\rightarrow I$ is called \textit{symmetric}, if $M\left(
x_{\sigma \left( 1\right) },\ldots,x_{\sigma \left( n\right) }\right) =M\left(
x_{1},\ldots,x_{n}\right) $ for each permutation $\sigma $ of $\left\{
1,\ldots,n\right\} $.

A mean $M:\left( 0,\infty \right) ^{n}\rightarrow \left( 0,\infty \right) $
is called \textit{homogeneous} if 
\[
M\left( tx_{1},\ldots,tx_{n}\right) = t M\left( x_{1},\ldots,x_{n}\right), \quad t,x_{1},\ldots,x_{n}>0.
\]

\begin{remark} \label{rem:1}
Let $J\subset \mathbb{R}$ be an interval, $\varphi :$ $J\rightarrow I$ be a
homeomorphic mapping of $J$ onto $I$. If $M:I^{n}\rightarrow I$ is a mean,
then the function $M^{\left[ \varphi \right] }:J^{n}\rightarrow J$ defined by
\[
M^{\left[ \varphi \right] }\left( x_{1},\ldots,x_{n}\right) :=\varphi
^{-1}\left( M\left( \varphi \left( x_{1}\right) ,\ldots,\varphi \left(
x_{n}\right) \right) \right), \left( x_{1},\ldots,x_{n}\right)
\in J^{n},
\]
is a mean in $J.$ The mean $M^{\left[ \varphi \right] }$ is called $\varphi $%
-conjugate of $M$. \\
Moreover, if $J=I$ and $M^{\left[ \varphi \right] }=M$, then $M$ is said to
be $\varphi $-self conjugate.
\end{remark}

Taking for $\varphi $ the reciprocal function, i.e. $\varphi \left( t\right)
=\frac{1}{t}$ for $t>0$, we get the following

\begin{remark} \label{rem:2}
A mean $M:\left( 0,\infty \right) ^{n}\rightarrow \left( 0,\infty \right) $
satisfies condition $(\ref{eq:1})$ if and only if, it is reciprocally self-conjugate.
\end{remark}
\noindent
Let us note the following obvious

\begin{remark} \label{rem:3}
Let $n\in \mathbb{N}$, $n\geq 2$. A mean $M:\left( 0,\infty \right)
^{n}\rightarrow \left( 0,\infty \right) $ satisfies $(\ref{eq:1})$ if and only if, the
exponentially conjugate mean $M^{\left[ \exp \right] }:\mathbb{R}%
^{n}\rightarrow \mathbb{R}$,%
\[
M^{\left[ \exp \right] }\left( t_{1},\ldots,t_{n}\right) :=\ln M\left(
e^{t_{1}},\ldots,e^{t_{n}}\right),\; t_{1},\ldots,t_{n}\in \mathbb{R},
\]
it is odd, that is
\[
M^{\left[ \exp \right] }\left( -t_{1},\ldots,-t_{n}\right) =-M^{\left[ \exp %
\right] }\left( t_{1},\ldots,t_{n}\right),\; t_{1},\ldots,t_{n}\in 
\mathbb{R}.
\]
\end{remark}%
\noindent
By the way let us note that there is no even mean.

\section{Weighted quasiarithmetic means} \label{sec:4}

Recall a description of the quasiarithmetic means, one of the most important
classes of means. These means are closely related to the conjugacy notion,
as they are conjugate to the weighted arithmetic means.

\begin{remark} [see for instance Chapter III in \cite{HLP}] \label{rem:4}
Let $\varphi :I\rightarrow \mathbb{R}$ be a continuous and strictly monotonic function, and let $%
p_{1},\ldots,p_{n}>0$ be such that $p_{1}+\ldots+p_{n}=1$. Then the function $%
\mathcal{A}_{p_{1},\ldots,p_{n}}^{\left[ \varphi \right] }:I^{n}\rightarrow I,$
given by 
\[
\mathcal{A}_{p_{1},\ldots,p_{n}}^{\left[ \varphi \right] }\left(
t_{1},\ldots,t_{n}\right) :=\varphi ^{-1}\left( \sum_{j=1}^{n}p_{j}\varphi
\left( t_{j}\right) \right),
\]
is a strict mean in $I,$ and it is called a \textit{weighted quasiarithmetic
mean; }the function $\varphi $ is referred to as its \textit{generator, }and 
$p_{1},\ldots,p_{n}$ as its weights.

$\mathcal{A}_{p_{1},\ldots,p_{n}}^{\left[ \varphi \right] }$ is symmetric iff $%
p_{1}=\ldots=p_{n}=\frac{1}{n};$ and then this mean, denoted by $\mathcal{A}^{%
\left[ \varphi \right] },$ is called \textit{quasiarithmetic}. Moreover, if $%
\varphi ,\psi :I\rightarrow \mathbb{R}$, then $\mathcal{A}%
_{p_{1},\ldots,p_{n}}^{\left[ \psi \right] }=\mathcal{A}_{p_{1},\ldots,p_{n}}^{%
\left[ \varphi \right] }$ if and only if $\psi =a\varphi +b$ for some real $%
a,b$, $a\neq 0.$
\end{remark}

\begin{remark}[\cite{HLP}, p.68] \label{rem:5}
Let $I=\left( 0,\infty \right) $. The following
conditions are equivalent:

\begin{itemize}

\item[$(i)$] the mean $\mathcal{A}_{p_{1},\ldots,p_{n}}^{\left[ \varphi \right] }$ is
homogeneous;

\item[$(ii)$] for some $a,b,r\in \mathbb{R}$, $a\neq 0$, 
\[
\varphi \left( t\right) =\left\{ 
\begin{array}{ccc}
a\,t^{r}+b & \mbox{if} & r\neq 0 \\ 
a\log t+b & \mbox{if} & r=0%
\end{array}%
\right. ,\quad t\in \left( 0,\infty \right) ;
\]

\item[$(iii)$] for some $r\in \mathbb{R}$,%
\[
	\mathcal{A}_{p_{1},\ldots,p_{n}}^{\left[ \varphi \right] }=\left\{ 
		\begin{array}{ccc}
		\left( \sum_{j=1}^{n}p_{j}x_{j}^{r}\right) ^{1/r} & \mbox{if} & r\neq 0 \\ 
		\prod\limits_{j=1}^{n}x_{j}^{p_{j}} & \mbox{if} & r=0%
		\end{array}%
	\right. ,\quad x_{1},\ldots,x_{n}\in \left( 0,\infty \right).
\]

\end{itemize}
\end{remark}
Applying the last result of Remark \ref{rem:1} we obtain the following

\begin{proposition} \label{prop:1}
Let $n\in \mathbb{N}$, $n\geq 2$. Suppose that $\varphi :\left( 0,\infty
\right) \rightarrow \mathbb{R}$ is a continuous strictly monotonic function,
and $p_{1},\ldots,p_{n}>0$ are such that $p_{1}+\ldots+p_{n}=1$. Then the
following two conditions are equivalent:
\begin{itemize}
\item[$(i)$] the weighted quasiarithmetic mean $\mathcal{A}_{p_{1},\ldots,p_{n}}^{\left[
\varphi \right] }:\left( 0,\infty \right) ^{n}\rightarrow \left( 0,\infty
\right)$ satisfies~$(\ref{eq:1})$, i.e.
\[
\mathcal{A}_{p_{1},\ldots,p_{n}}^{\left[ \varphi \right] }\left(
x_{1},\ldots,x_{n}\right) \mathcal{A}_{p_{1},\ldots,p_{n}}^{\left[ \varphi \right]
}\left( \frac{1}{x_{1}},\ldots,\frac{1}{x_{n}}\right) =1, \quad x_{1},\ldots,x_{n}>0;
\]
\item[$(ii)$] there are \thinspace $a,b\in \mathbb{R}$, $a\neq 0$ such that $%
\varphi $ satisfies the functional equation 
\[
\varphi \left( \frac{1}{x}\right) =a\varphi \left( x\right) +b, \quad x\in \left( 0,\infty \right) .
\]

\end{itemize}

\end{proposition}

\begin{proof}
Note that (\ref{eq:1}) holds if and only if 
\[
\mathcal{A}_{p_{1},\ldots,p_{n}}^{\left[ \varphi \right] }\left(
x_{1},\ldots,x_{n}\right) =\frac{1}{\mathcal{A}_{p_{1},\ldots,p_{n}}^{\left[
\varphi \right] }\left( \frac{1}{x_{1}},\ldots,\frac{1}{x_{n}}\right) }, \quad x_{1},\ldots,x_{n}>0,
\]
that is, if and only if 
\[
\mathcal{A}_{p_{1},\ldots,p_{n}}^{\left[ \varphi \right] }=\mathcal{A}%
_{p_{1},\ldots,p_{n}}^{\left[ \varphi \circ \beta \right] },
\]
where $\beta :\left( 0,\infty \right) \rightarrow \left( 0,\infty \right) $
denotes the reciprocal function $\beta \left( x\right) =\frac{1}{x}$. In
view of Remark \ref{rem:1} (see also \cite{HLP}, p.66), this equality holds if and
only if \thinspace there are $a,b\in \mathbb{R}$, $a\neq 0$ such that $%
\left( \varphi \circ \beta \right) \left( x\right) =a\varphi \left( x\right)
+b$ \ for all \ $x\in \left( 0,\infty \right) $, that is if, and only if (2)
holds.
\end{proof}

\section{Odd weighted quasiarithmetic means and main results} \label{sec:5}

By Remark \ref{rem:3}, a mean $M$ on $\left( 0,\infty \right) $ satisfies condition
(\ref{eq:1}) iff the mean $M^{\left[ \exp \right] }$ is odd on $\mathbb{R}.$
Stimulated by this fact we prove

\begin{proposition} \label{prop:2}
Let $n\in \mathbb{N}$, $n\geq 2$ and let the positive numbers $%
p_{1},\ldots,p_{n}$ such that $p_{1}+\ldots+p_{n}=1$ be fixed. Suppose that $%
\varphi :\mathbb{R\rightarrow R}$ is a continuous strictly monotonic function%
$.$ \\
Then the following conditions are equivalent:
\begin{itemize}
\item[$(i)$] \ \ the quasiarithmetic weighted mean $\mathcal{A}_{p_{1},\ldots,p_{n}}^{%
\left[ \varphi \right] }:\mathbb{R}^{n}\rightarrow \mathbb{R}$ is odd;
\item[$(ii)$] \ the function $\varphi -\varphi \left( 0\right) $ is odd or,
equivalently, 
\[
\varphi \left( -t\right) +\varphi \left( t\right) =2\varphi \left( 0\right) 
, \quad t\in \mathbb{R}.
\]
\end{itemize}
\end{proposition}

\begin{proof}
To prove the implication (i)$\Longrightarrow $(ii), assume that $\mathcal{A}%
_{p_{1},\ldots,p_{n}}^{\left[ \varphi \right] }$ is an odd function. By the
definition of $\mathcal{A}_{p_{1},\ldots,p_{n}}^{\left[ \varphi \right] }$ it
means that%
\[
\varphi ^{-1}\left( \sum_{j=1}^{n}p_{j}\varphi \left( -t_{j}\right) \right)
=-\varphi ^{-1}\left( \sum_{j=1}^{n}p_{j}\varphi \left( t_{j}\right) \right) 
, \quad t_{1},\ldots,t_{n}\in \mathbb{R}.
\]

Choosing arbitrarily $j_{0}\in \left\{ 2,\ldots,n-1\right\} $ and setting%
\[
p:=p_{1}+\ldots+p_{j_{0}},
\]
\[
t_{1}=\ldots=t_{j_{0}}=s, \quad t_{j_{0}+1}=\ldots=t_{n}=t,
\]
we hence get%
\[
\varphi ^{-1}\left( p\varphi \left( -s\right) +\left( 1-p\right) \varphi
\left( -t\right) \right) =-\varphi ^{-1}\left( p\varphi \left( s\right)
+\left( 1-p\right) \varphi \left( t\right) \right) , \quad s,t\in 
\mathbb{R}.
\]
For $s=\varphi ^{-1}\left( u\right) $, $t=\varphi ^{-1}\left( v\right) $,
we\ have 
\[
\varphi ^{-1}\left( p\varphi \left( -\varphi ^{-1}\left( u\right) \right)
+\left( 1-p\right) \varphi \left( -\varphi ^{-1}\left( v\right) \right)
\right) =-\varphi ^{-1}\left( pu+\left( 1-p\right) v\right) , \quad %
u,v\in \varphi \left( \mathbb{R}\right) .
\]
and, taking $\varphi $ of both sides%
\[
p\varphi \left( -\varphi ^{-1}\left( u\right) \right) +\left( 1-p\right)
\varphi \left( -\varphi ^{-1}\left( v\right) \right) =\varphi \left(
-\varphi ^{-1}\left( pu+\left( 1-p\right) v\right) \right) , \quad %
u,v\in \varphi \left( \mathbb{R}\right) .
\]
Hence, setting 
\begin{equation}\label{eq:2}
\psi :=\varphi \circ \left( -\varphi ^{-1}\right),
\end{equation}
we get 
\[
\psi \left( pu+\left( 1-p\right) v\right) =p\psi \left( u\right) +\left(
1-p\right) \psi \left( v\right) , \quad u,v\in \varphi \left( 
\mathbb{R}\right) .
\]
This equation and the Dar\'{o}czy-Páles identity (see \cite{DarPal})%
\[
p\left( p\frac{u+v}{2}+\left( 1-p\right) u\right) +\left( 1-p\right) \left(
pv+\left( 1-p\right) \frac{u+v}{2}\right) =\frac{u+v}{2},%
x,y\in \mathbb{R}
\]
easily imply that $\psi $ satisfies the Jensen functional equation%
\[
\psi \left( \frac{u+v}{2}\right) =\frac{\psi \left( u\right) +\psi \left(
v\right) }{2}, \quad u,v\in \varphi \left( \mathbb{R}\right) .
\]
Since $\psi $ is continuous in the interval $\varphi \left( \mathbb{R}%
\right) $, in view of Theorem 1, p. 315 in M. Kuczma \cite{Kuczma}, there
are some constant $a,b\in \mathbb{R}$, $a\neq 0$ such that 
\[
\psi \left( u\right) =au+b, \quad u\in \varphi \left( \mathbb{%
R}\right) .
\]
Hence, by (\ref{eq:2}), 
\[
\varphi \circ \left( -\varphi ^{-1}\right) \left( u\right) =au+b, \quad u\in \varphi \left( \mathbb{R}\right) ,
\]
whence, setting $u=\varphi \left( t\right) $, we get%
\begin{equation} \label{eq:3}
\varphi \left( -t\right) =a\varphi \left( t\right) +b, \quad %
t\in \mathbb{R}.
\end{equation}%
Replacing here first $t$ by $-t$, and then applying this equation again, we
get%
\begin{eqnarray*}
\varphi \left( t\right) &=&a\varphi \left( -t\right) +b=a\left[ a\varphi
\left( t\right) +b\right] +b \\
&=&a^{2}\varphi \left( t\right) +ab+b,
\end{eqnarray*}
whence%
\begin{equation} \label{eq:4}
\left( 1-a^{2}\right) \varphi \left( t\right) =ab+b, \quad t\in \mathbb{R}.
\end{equation}%
Since the right side is constant, it follows that $\ a^{2}=1$. Consequently,
either $a=1$ or $a=-1\,.$

If $a=1$ then by (\ref{eq:4}) the number $b$ must be $0$ and, by (\ref{eq:3}) we get 
\[
\varphi \left( -t\right) =\varphi \left( t\right) , \quad t\in 
\mathbb{R},
\]
that is a contradiction, as the function $\varphi $ is strictly monotonic.

If $a=-1$ then $b$ in (\ref{eq:4}) can be arbitrary and (\ref{eq:3}) becomes 
\[
\varphi \left( -t\right) =-\varphi \left( t\right) +b, \quad %
t\in \mathbb{R}.
\]
Setting here $t=0$ gives $\varphi \left( 0\right) =-\varphi \left( 0\right)
+b$, whence $b=2\varphi \left( 0\right) .$ Thus, by (\ref{eq:3}), 
\[
\varphi \left( -t\right) -\varphi \left( 0\right) =-\left[ \varphi \left(
t\right) -\varphi \left( 0\right) \right] , \quad t\in \mathbb{R}%
, 
\]

so the function $\varphi -\varphi \left( 0\right) $ is odd, which proves the
implication (i)$\Longrightarrow $(ii).

To prove the reversed implication assume that $\varphi -\varphi \left(
0\right) $ is odd. Then, the inverse function $\varphi ^{-1}\left( y+\varphi
\left( 0\right) \right) $ is also odd. Applying these facts in turn, by the
definition of $\mathcal{A}_{p_{1},\ldots,p_{n}}^{\left[ \varphi \right] }$, we
have 
\begin{eqnarray*}
\mathcal{A}_{p_{1},\ldots,p_{n}}^{\left[ \varphi \right] }\left(
-t_{1},\ldots,-t_{n}\right) &=&\varphi ^{-1}\left( \sum_{j=1}^{n}p_{j}\varphi
\left( -t_{j}\right) \right) \\
& = & \varphi ^{-1}\left( \sum_{j=1}^{n}p_{j}\left[
\varphi \left( -t_{j}\right) -\varphi \left( 0\right) \right] +\varphi
\left( 0\right) \right) \\
&=&\varphi ^{-1}\left( -\sum_{j=1}^{n}p_{j}\left[ \varphi \left(
t_{j}\right) -\varphi \left( 0\right) \right] +\varphi \left( 0\right)
\right) \\
&=&-\varphi ^{-1}\left( \sum_{j=1}^{n}p_{j}\left[ \varphi \left(
-t_{j}\right) -\varphi \left( 0\right) \right] +\varphi \left( 0\right)
\right) \\
&=&-\varphi ^{-1}\left( \sum_{j=1}^{n}p_{j}\varphi \left( t_{j}\right)
\right) =-\mathcal{A}_{p_{1},\ldots,p_{n}}^{\left[ \varphi \right] }\left(
t_{1},\ldots,t_{n}\right)
\end{eqnarray*}
for all $t_{1},\ldots,t_{n}\in \mathbb{R}$, which completes the proof.
\end{proof}

\section{Main results} \label{sec:6}

We begin this section with the following

\begin{remark} \label{rem:6}
Let $n\in \mathbb{N}$, $n\geq 2,$ and $J\subset \left( 0,\infty \right) $ be
a fixed interval. If $\gamma :J\rightarrow \left( 0,\infty \right) $ is a
continuous strictly monotonic function, and the real numbers $%
p_{1},\ldots,p_{n}>0$ are such that $p_{1}+\ldots+p_{n}=1,$ then

\begin{itemize}

\item[$(i)$] the function $\mathcal{M}_{p_{1},\ldots,p_{n}}^{\left[ \gamma \right]
}:I^{n}\rightarrow I$ given by 
\[
\mathcal{M}_{p_{1},\ldots,p_{n}}^{\left[ \gamma \right] }\left(
x_{1},\ldots,x_{n}\right) :=\gamma ^{-1}\left( \prod\limits_{j=1}^{n}\left[
\gamma \left( x_{j}\right) \right] ^{p_{j}}\right)
\]
is a strict mean in $I$\textit{; }
\item[$(ii)$] \ $\mathcal{M}_{p_{1},\ldots,p_{n}}^{\left[ \gamma \right] }=\mathcal{A}%
_{p_{1},\ldots,p_{n}}^{\left[ \log \circ \gamma \right] }.$
\end{itemize}

\end{remark}

\begin{proof}
Indeed, for arbitrary $x_{1},\ldots.,x_{n}>0$, 
\begin{eqnarray*}
\mathcal{M}_{p_{1},\ldots,p_{n}}^{\left[ \gamma \right] }\ \left(
x_{1},\ldots,x_{n}\right) &=&\gamma ^{-1}\left( \prod\limits_{j=1}^{n}\left[
\gamma \left( x_{j}\right) \right] ^{p_{j}}\right) =\gamma ^{-1}\left( \exp
\left( \log \prod\limits_{j=1}^{n}\left[ \gamma \left( x_{j}\right) \right]
^{p_{j}}\right) \right) \\
&=&\left( \log \circ \gamma \right) ^{-1}\left( \sum_{j=1}^{n}p_{j}\left(
\log \circ \gamma \right) \left( x_{j}\right) \right) \\
&=&\mathcal{A}_{p_{1},\ldots,p_{n}}^{\left[ \log \circ \gamma \right] }\ \left(
x_{1},\ldots,x_{n}\right) ,
\end{eqnarray*}
so (ii) holds true. Since $\mathcal{A}_{p_{1},\ldots,p_{n}}^{\left[ \log \circ
\gamma \right] }$ is a quasiarithmetic mean, condition (ii) implies (i).
\end{proof}
By analogy to the quasiarithmetic means, the mean \ $\mathcal{M}%
_{p_{1},\ldots,p_{n}}^{\left[ \gamma \right] }$ could be called a\textit{\
weighted quasigeometric mean,} the function $\gamma $ its \textit{generator, 
}and $p_{1},\ldots,p_{n},$ the weights.

This remark, Remark \ref{rem:3} and Proposition \ref{prop:1} give the following characterization
of weighted quasiarithmetic (or weighted quasigeometric) means on $\left(
0,\infty \right) $ and satisfying condition (\ref{eq:1}).

\begin{theorem} \label{the:1}
Let $n\in \mathbb{N}$, $n\geq 2,$ and let the positive numbers $%
p_{1},\ldots,p_{n}$ such that $p_{1}+\ldots+p_{n}=1$ be fixed. Suppose that $%
\gamma :\left( 0,\infty \right) \rightarrow \left( 0,\infty \right) $ is
continuous and strictly monotonic. Then the following conditions are
equivalent:
\begin{itemize}
\item[$(i)$] \ \ the mean $\mathcal{M}_{p_{1},\ldots,p_{n}}^{\left[ \gamma \right]
}:\left( 0,\infty \right) \rightarrow \left( 0,\infty \right) \mathbb{\ }$\
satisfies condition%
\[
\mathcal{M}_{p_{1},\ldots,p_{n}}^{\left[ \gamma \right] }\left(
x_{1},\ldots,x_{n}\right) \mathcal{M}_{p_{1},\ldots,p_{n}}^{\left[ \gamma \right]
}\left( \frac{1}{x_{1}},\ldots,\frac{1}{x_{n}}\right) =1, \quad %
x_{1},\ldots,x_{n}>0;
\]
\item[$(ii)$] \ the generator $\gamma $ satisfies the functional equation%
\[
\gamma \left( x\right) \gamma \left( \frac{1}{x}\right) =\left[ \gamma
\left( 1\right) \right] ^{2}, \quad x>0.
\]
\end{itemize}
\end{theorem}
The main result of this paper reads as follows:

\begin{theorem} \label{the:2}
Let $n\in \mathbb{N}$, $n\geq 2$ and let the positive numbers $%
p_{1},\ldots,p_{n}$ be such that $p_{1}+\ldots+p_{n}=1$. Suppose that $\gamma
:\left( 0,\infty \right) \rightarrow \left( 0,\infty \right) $ is continuous
and strictly monotonic. Then the following conditions are equivalent:
\begin{itemize}
\item[$(i)$] the mean $\mathcal{M}_{p_{1},\ldots,p_{n}}^{\left[ \gamma \right]
}:\left( 0,\infty \right) ^{n}\rightarrow \left( 0,\infty \right) $ is
homogeneous in $\left( 0,\infty \right) $ and 
\[
\mathcal{M}_{p_{1},\ldots,p_{n}}^{\left[ \gamma \right] }\left(
x_{1},\ldots,x_{n}\right) \mathcal{M}_{p_{1},\ldots,p_{n}}^{\left[ \gamma \right]
}\left( \frac{1}{x_{1}},\ldots,\frac{1}{x_{n}}\right) =1,\ \ \ \ \
x_{1},\ldots,x_{n}>0;
\]

\item[$(ii)$] \ there are positive $a,b\in \mathbb{R}$, $a\neq 0$ such that 
\[
\gamma \left( t\right) =e^{b}t^{a}, \quad t>0;
\]

\item[$(iii)$] $\mathcal{M}_{p_{1},\ldots,p_{n}}^{\left[ \gamma \right] }=\mathcal{G}%
_{p_{1},\ldots,p_{n}}$, where 
\[
\mathcal{G}_{p_{1},\ldots,p_{n}}\left( x_{1},\ldots,x_{n}\right)
:=x_{1}^{p_{1}}\cdot \ldots\cdot x_{n}^{p_{n}}, \quad %
x_{1},\ldots,x_{n}>0,
\]
is the weighted geometric mean.
\end{itemize}
\end{theorem}

\begin{proof}
Assume (i). By Remark \ref{rem:6}(ii) we have $\mathcal{M}_{p_{1},\ldots,p_{n}}^{\left[
\gamma \right] }=\mathcal{A}_{p_{1},\ldots,p_{n}}^{\left[ \log \circ \gamma %
\right] }$. \ The homogeneity of the weighted quasiarithmetic mean $\mathcal{%
A}_{p_{1},\ldots,p_{n}}^{\left[ \log \circ \gamma \right] }$ implies (see \cite{HLP} p.~68) that there are $a,b,r\in \mathbb{R}$, $a\neq 0$ such that 
\[
\log \circ \gamma \left( t\right) =\left\{ 
\begin{array}{ccc}
at^{r}+b & \mbox{if} & r\neq 0 \\ 
a\log t+b & \mbox{if} & r=0%
\end{array}%
\right. , \quad t>0.
\]
If $r\neq 0$ then 
\[
\gamma \left( t\right) =b\exp \left( at^{r}\right) , \quad t>0,
\]
whence 
\[
\gamma \left( t\right) \gamma \left( \frac{1}{t}\right) =b\exp \left(
at^{r}\right) b\exp \left( a\left( \frac{1}{t}\right) ^{r}\right) =b^{2}\exp
\left( a\left( t^{r}+t^{-r}\right) \right) ,
\]
so the function $\left( 0,\infty \right) \ni t\rightarrow \gamma \left(
t\right) \gamma \left( \frac{1}{t}\right) $ is not constant. \newline
If $r=0$ then%
\[
\gamma \left( t\right) =e^{b}t^{a}, \quad t>0,
\]
whence%
\[
\gamma \left( t\right) \gamma \left( \frac{1}{t}\right) =\left(
e^{b}t^{a}\right) \left( e^{b}t^{-a}\right) =\left( e^{b}\right) ^{2}=\left[
\gamma \left( 1\right) \right] ^{2}, \quad t>0.
\]
Now (ii) follows from Theorem \ref{the:1}. \newline
Assume (ii). Then 
\[
\gamma ^{-1}\left( u\right) =e^{-\frac{b}{a}}u^{\frac{1}{a}}, \quad %
u\in \gamma \left( 0,\infty \right) ,
\]
and, taking into account that $p_{1}+\ldots+p_{n}=1,$ we have, for all $%
x_{1},\ldots,x_{n}>0,$%
\begin{eqnarray*}
\mathcal{M}_{p_{1},\ldots,p_{n}}^{\left[ \gamma \right] }\left(
x_{1},\ldots,x_{n}\right) &=&\gamma ^{-1}\left( \prod\limits_{j=1}^{n}\left[
\gamma \left( x_{j}\right) \right] ^{p_{j}}\right) =e^{-\frac{b}{a}}\left(
\prod\limits_{j=1}^{n}\left( e^{b}x_{j}^{a}\right) ^{p_{j}}\right) ^{1/a} \\
&=&\prod\limits_{j=1}^{n}x_{j}^{p_{j}}=\mathcal{G}_{p_{1},\ldots,p_{n}}\left(
x_{1},\ldots,x_{n}\right) ,
\end{eqnarray*}
so (iii) holds true. \newline
Since the implication $(iii)\Longrightarrow (i)$ is obvious, the proof is
complete.
\end{proof}

\begin{remark} \label{rem:7}
Let us note that the main results of this paper can be also extended to the
class generalized weighted quasiarithmetic means of the form%
\[
\mathcal{M}^{\left[ g_{1},\ldots,g_{n}\right] }\left( x_{1},\ldots,x_{n}\right)
:=\left( \prod\limits_{j=1}^{n}g_{j}\right) ^{-1}\left(
\prod\limits_{j=1}^{n}g_{j}\left( x_{j}\right) \right)
\]
where $g_{1},\ldots,g_{n}:\left( 0,\infty \right) \rightarrow \left( 0,\infty
\right) $ are continuous strictly increasing (or strictly decreasing)
functions and $\prod\limits_{j=1}^{n}g_{j}$ is a single variable function
being the product of $g_{1}\ldots g_{n}$ (see \cite{JM2010}). Also in this case,
if $\mathcal{M}^{\left[ g_{1},\ldots,g_{n}\right] }$\ is homogeneous and
satisfies (1), then $\mathcal{M}^{\left[ g_{1},\ldots,g_{n}\right] }=\mathcal{G}%
_{p_{1},\ldots,p_{n}}$ for some positive $p_{1},\ldots,p_{n}$ such that $%
p_{1}+\ldots+p_{n}=1$.
\end{remark}
We conclude this paper with the following

\begin{remark} \label{rem:8}
In the family of all homogeneous weighted quasiarithmetic means, only the
geometric mean is self reciprocally conjugate, i.e. satisfies condition~\mbox{(\ref{eq:1})}.
\end{remark}


\end{document}